\documentclass{amsart}
\usepackage{times}
\usepackage{amsthm}
\usepackage{amssymb}
\usepackage[pdftex]{graphicx}
\usepackage[all]{xy}
\usepackage[mathscr]{eucal}
\usepackage{amsmath,latexsym,oldlfont}
\usepackage{mathdots}
\usepackage{pifont}
\usepackage{stmaryrd}
\usepackage{textcomp}
\usepackage{pifont}
\usepackage{multirow}
\usepackage[colorlinks,linkcolor=blue,citecolor=red,linktocpage=true]{hyperref}
\usepackage{mathtools}
\numberwithin{equation}{section} 
\newtheorem{lem}{Lemma}

\newtheorem{prop}[lem]{Proposition}

\theoremstyle{definition}

\theoremstyle{remark}
\newtheorem{rem}{Remark}

\newcommand{\swedge}{{\scriptstyle\wedge}}

\newcommand{\nwedge}{\mathchoice{{\textstyle\wedge}}%
    {{\wedge}}%
    {{\textstyle\wedge}}%
    {{\scriptstyle\wedge}}}
\newsavebox{\spacebox}
\begin{lrbox}{\spacebox}
    \verb*! !
\end{lrbox}

\title[Note on the codimension of  the linear section of $L (6, 12)$]{A note on the codimension of  the linear section of the Lagrangian-Grassmannian $L (6, 12)$}
\author[J. Carrillo--Pacheco]{Jes\'us Carrillo--Pacheco}
\author[F. Jarqu\'{\i}n--Z\'arate]{Fausto Jarqu\'{\i}n--Z\'arate}
\author[M. Velasco--Fuentes]{Maurilio Velasco--Fuentes}
\author[F. Zaldivar]{Felipe Zaldivar}
\address[J. Carrillo--Pacheco, F. Jarqu\'{\i}n--Z\'arate, M. Velasco--Fuenetes]{Academia de Matem\'aticas, Universidad Aut\'onoma de la Ciudad de M\'exico, 09390 M\'exico, D. F., M\'exico.}
\address[F. Zaldivar]{Departamento de Matem\'aticas, Universidad Aut\'onoma Metropolitana-I, 09340 M\'exico, D. F., M\'exico.}
\email[J. Carrillo--Pacheco]{jesus.carrillo@uacm.edu.mx}
 \email[F. Jarqu\'{\i}n--Z\'arate]{fausto.jarquin@uacm.edu.mx}
 \email[M. Velasco--Fuentes]{maurilio.velasco.fuentes@uacm.edu.mx}
 \email[F. Zaldivar]{fz@xanum.uam.mx}
 \thanks{Jes\'us Carrillo-Pacheco, Fausto Jarqu\'{\i}n-Z\'arate and
Maurilio Velasco-Fuentes are supported by the Cryptography
Laboratory Project PI2013-29. Secretar{\'\i}a de Ciencia,
Tecnolog{\'\i}a e Innovaci\'on del Distrito Federal (SECITI),
M\'exico.}

\begin{document}



\begin{abstract}
{Consider a $2n$-dimensional symplectic vector space $E$
    over an arbitrary field $\mathbb{F}$. Given a contraction map
    $f: \wedge^n E \rightarrow \wedge^{n-2} E$ such that the Lagrangian--Grassmannian
    $L(n,2n)=G(n,2n)\cap{\mathbb P}(\ker f)$, where $\wedge^r E$ denotes the $r$-th exterior
    power of $E$ and ${\mathbb P}(\ker f)$ is the projectivization of $\ker f$.
    In this paper, for a symplectic vector space $E$ of dimension $n=6$,
    we prove that the surjectivity of the contraction map $f:\wedge^{6} E \rightarrow \wedge^{4} E$
    depends on the   characteristic of the base field and we calculate the codimension of the
    linear section ${\mathbb P}(\ker f)\subseteq {\mathbb P}(\nwedge^{6}E)$ for any characteristic.}
\end{abstract}

 \maketitle

\section{Introduction}\label{Intro}

Let $E$ be  a $2n$-dimensional symplectic vector space over an
arbitrary field $\mathbb{F}$ with symplectic form $\langle\;
,\;\rangle$. Consider  {\it the contraction map}
$f:\nwedge^nE\rightarrow \nwedge^{n-2}E$
given by
\begin{gather}\label{mapf}
f(w_1\swedge\cdots\swedge w_n)=\sum_{1\leq s<t\leq n}\langle
w_s,w_t\rangle w_1\swedge\cdots\swedge
\widehat{w}_s\swedge\cdots\swedge
\widehat{w}_t\swedge\cdots\swedge w_n,
\end{gather}
where $\widehat{w}$ means that the corresponding term is omitted.
Our main result shows that, in general, the map $f$ is not
surjective. Since,  by \cite{bib2} the Lagrangian-Grassmannian variety $L(n,2n)$ is cut out by the projectivization ${\mathbb P}(\ker f)$ of the kernel of $f$, it follows that   the codimension of $L(n,2n)$ in its Pl\"ucker embedding
 is not $C^{2n}_n-C^{2n}_{n-2}$, where $C_n^m$ denotes the binomial
 coefficient. Specifically, we prove that for $n=6$, and a field of characteristic $3$, the contraction map $f:\nwedge^6 E\rightarrow \nwedge^{4}E$ given by \eqref{mapf}   is not surjective. To prove this, we use a combinatorial description in Lemma \ref{Part} of the set of indices that label the   Pl\"ucker linear relations that is then used to
 describe the linear section ${\mathbb P}(\ker f)$ that cuts out the Lagrangian-Grassmannian $L(6,12)$ in the Grassmannian variety $G(6,12)$ in any
 characteristic. As a consequence we show that the codimension of $L(6,12)$ in its Pl\"ucker embedding
 depends of the characteristic of the base field.

The paper is organized as follows. In Section $\ref{PreTer}$ we recall some results of the contraction map
\eqref{mapf} and the Lagrangian--Grassmannian. In Section
$\ref{EjemP}$ we give an explicit example where the
contraction map is not surjective and give all the details
involved  to obtain the linear section ${\mathbb P}(\ker f)$ that defines $L(6,12)$.

\section{Preliminaries}\label{PreTer}

Let  $E$ be a $2n$--dimensional vector space over $\mathbb{F}$
equipped with a non-degenerate symplectic form $\langle\;,
\;\rangle$. Define the set $I(\ell,m)=\{\alpha=(\alpha_1,\ldots,\alpha_{\ell}) :
1\leq\alpha_1<\cdots<\alpha_{\ell}\leq m\}$, such that $\alpha_i \in
\mathbb{N}$, and  the support of $\alpha=(\alpha_1,\ldots,\alpha_{\ell}) \in I(\ell,m)$  as
the set $\text{supp}(\alpha):=\{\alpha_1,\ldots,\alpha_{\ell}\}$. Thus, all indices $\alpha$ are ordered sets of $\ell$ different integers in the set $\{1,2,\ldots,m\}$. In what follows, all indices $\alpha$ are sets with $\ell$ different elements  in the set
$\{1,2,\ldots,m\}$, and up to permutation, we may (and do so) think of them in $I(\ell,m)$.

Choose a basis $\{e_1,\ldots,e_{2n}\}$ of the  symplectic space $E$
such that
$$
\langle e_i,e_j\rangle=\begin{cases}
1& \text{if $j=2n-i+1$}, \\
0 & \text{otherwise}.
\end{cases}$$
Then, for $\alpha=(\alpha_1,\ldots,\alpha_n)\in I(n,2n)$ write
\begin{align*}
e_{\alpha}&:=e_{\alpha_1}\swedge\cdots\swedge e_{\alpha_n},\\
e_{\alpha_{st}}&:=e_{\alpha_1}\swedge\cdots\swedge\widehat{e}_{\alpha_s}
\swedge\cdots\swedge\widehat{e}_{\alpha_t}\swedge\cdots   \swedge e_{\alpha_n},\\
p_{i\alpha_{st}(2n-i+1)}&:=p_{i\alpha_1 \cdots \widehat{\alpha}_s
\cdots \widehat{\alpha}_t \cdots \alpha_n(2n-i+1)},
\end{align*}
where $\widehat{e}_{{\alpha}_k}$ and
$\widehat{\alpha}_k$ means that the corresponding term is omitted.
Denote  by $\nwedge^nE$ the $n$-th exterior power of $E$,
which is  generated by $\{ e_{\alpha}: \alpha \in
I(n,2n)\}$. For $w=\sum_{\alpha\in I(n,2n)}p_{\alpha}e_{\alpha} \in
\nwedge^nE$, the coefficients $p_{\alpha}$ are the {\it Pl\"ucker
coordinates of} $w$. In \cite[Proposition
6]{bib2} the kernel of the contraction map $f$ is characterized as follows: For $w=\sum_{\alpha\in I(n,2n)}p_{\alpha}e_{\alpha}\in
\nwedge^nE$ written in Pl\"ucker coordinates,  we have that
$$w\in\ker f\iff \sum_{i=1}^np_{i\alpha_{st}(2n-i+1)}=0,\;\text{for all $\alpha_{st}\in I(n-2,2n)$}.$$

In \cite[Section 3]{bib2} these linear forms were given the following
description: For $\alpha_{st}\in I(n-2,2n)$ define the linear
polynomials
$$\Pi_{\alpha_{st}}:=\sum_{i=1}^nc_{i,\alpha_{st},2n-i+1}X_{i,\alpha_{st},2n-i+1},$$
with
$$c_{i,\alpha_{st},2n-i+1}=\begin{cases}
1  & \text{if $|\text{supp}\{i,\alpha_{st},2n-i+1\}|=n$}, \\
0 & \text{otherwise},
\end{cases}$$
hence $\Pi_{\alpha_{st}}$ are polynomials in the ring ${\mathbb F}[X_{\alpha}:\alpha\in I(n,2n)]$. From the following
formula in Pl\"ucker coordinates
\begin{equation}\label{eq2.1}
 X_{1\sqcup 2n}+X_{2\sqcup (2n-1)}+\cdots+X_{n\sqcup (n+1)}=0,
  \end{equation}
where the symbols $\sqcup$ are to be replaced by elements
$\alpha_{st}\in I(n-2,2n)$,    we obtain
homogeneous linear equations, that we call a {\it Pl\"ucker linear relations} in $k$-variables,
\begin{gather}\label{Eqxx}
\Pi_{\alpha_{st}}:=\; X_{1,\alpha_{st},  2n}+X_{2,\alpha_{st},
(2n-1)}+\cdots+X_{n,\alpha_{st},  (n+1)}=0
\end{gather}
where the term $X_{i,\alpha_{st},  (2n-i+1)}$ does
not appear if $|\text{supp}\{i,\alpha_{st},  (2n-i+1)\}|<n$. When
this happens we say $\Pi_{\alpha_{st}}$ is a $k$-plane.
For the system of homogeneous linear equations $\Pi_{\alpha_{st}}$,
$\alpha_{st}\in I(n-2,2n)$, we denote by $B$ its associated matrix.
 Clearly the matrix $B$ is of order  $C^{2n}_{n-2}\times C^{2n}_{n}$. For example, if $n=6$ formula \eqref{Eqxx} becomes
 \begin{equation}\label{eq2.4}
 X_{1\sqcup C}+X_{2\sqcup B}+X_{3\sqcup A}+X_{4\sqcup 9}+X_{5\sqcup 8}+X_{6\sqcup 7}=0,
 \end{equation}
 where $A=10$, $B=11$, $C=12$.

Recall that a vector subspace $W$ of $E$ is \emph{isotropic} iff for all
$x,y \in W$ we have that $\langle\; x, y\;\rangle=0$, and if $W$ is isotropic
its dimension is at most $n$. The {\it Lagrangian-Grassmannian}
$L(n,2n)$ is the projective variety given by the isotropic vector
subspaces  $W\subseteq E$ of maximal dimension $n$:
$$L(n,2n)=\{W\in G(n,2n): W\;\text{is isotropic and $n$-dimensional}\},$$
where $G(n,2n)$ denotes the Grassmannian variety of vector subspaces
of dimension $n$ of $E$. The {\it Pl\"ucker embedding} is the
regular map $\rho:G(n,2n)\rightarrow {\mathbb P}(\wedge^nE)$
given on each $W\in G(n,2n)$ by choosing  a basis $w_1,\ldots,
w_n$ of $W$ and then mapping the vector subspace $W\in G(n,2n)$ to
the tensor $w_1\swedge\cdots\swedge w_n\in \nwedge^nE$. Since
choosing a different basis  of $W$ changes the tensor
$w_1\swedge\cdots\swedge w_n$ by a nonzero scalar, this tensor is a
well-defined element in the projective space ${\mathbb
P}(\nwedge^nE)\simeq{\mathbb P}^{N-1}$, where
$N=C^{2n}_n=\dim_{\mathbb{F}}(\nwedge^nE)$. Under the Pl\"ucker
embedding, the Lagrangian-Grassmannian is given by
$$L(n,2n)=\{w_1\swedge\cdots\swedge w_n\in G(n,2n): \langle w_i,w_j\rangle=0\;\text{for all $1\leq i<j\leq n$}\}.$$
Using  the contraction map  $f:\nwedge^nE\rightarrow
\nwedge^{n-2}E$ given by \eqref{mapf}, if ${\mathbb P}(\ker f)$ is
the projectivization of $\ker f$, in \cite{bib2} it is proved that $L(n,2n)=G(n,2n)\cap{\mathbb P}(\ker f)$.
We call ${\mathbb P}(\ker f)$ the linear section that defines $L(n,2n)$ in ${\mathbb P}(\nwedge^nE)$.
\section{Non surjectivity of the contraction map in $\text{\rm char}({\mathbb F}) =3$}\label{EjemP}

The purpose of this section is twofold: First, to provide a description, completely
explicit and self-contained of the linear space
${\mathbb{P}}(\ker f)$ for $n=6$, for any field
$\mathbb{F}$, and then using this characterization we give an example of the non surjectiviy of the contraction map for a field of characteristic $3$.

Let $P_1=(1,C), P_2=(2,B), P_3=(3,A), P_4=(4,9), P_5=(5,8), P_6=(6,7)$,
 where $A=10, B=11, C=12$, as in Section \ref{PreTer},
 $\Sigma_6
=\{P_1,P_2,\ldots,P_6\}$, and $C_2(\Sigma_6)$ the set of all
combinations of $6$ objects taken 2 at a time. For
$1\leq\alpha_1<\alpha_2\leq 12$ such that $\alpha_1+\alpha_2\neq
13$, we define the following set
\begin{gather*}
\Sigma\{\alpha_1,\alpha_2\} = \{(\alpha_1,\alpha_2, P_i)\in I(4,12) : \; i +
\alpha_j \neq 13, \alpha_j+13-i \neq 13, j= 1,2\}.
\end{gather*}

Now, for $1\leq \alpha_1< \alpha_2 < \alpha_3 <\alpha_4 \leq 12$
such that $\alpha_i + \alpha_j \neq 13$, define
\begin{gather*}
\Sigma\{\alpha_1,\alpha_2,\alpha_3,\alpha_4\} =
\{(\alpha_1,\alpha_2,\alpha_3,\alpha_4)\in I(4,12) : \alpha_i +\alpha_j \neq 13
\mbox{ with } 1\leq i,j \leq 4\}.
\end{gather*}

\begin{lem}\label{Part}
With the notation above  we have a partition of $I(4,12)$, given by
    $$  C_2(\Sigma_6) \cup \big( \bigcup_{1\leq \alpha_1< \alpha_2 \leq
        12 \atop \alpha_1 +\alpha_2 \neq 13}\Sigma
    \{\alpha_1,\alpha_2\}\big) \cup \big(\bigcup_{1 \leq \alpha_1 <
        \alpha_2 < \alpha_3 < \alpha_4 \leq 12 \atop \alpha_i +\alpha_j
        \neq 13} \Sigma\{\alpha_1,\alpha_2,\alpha_3,\alpha_4\} \big).$$
\end{lem}
\begin{proof}
    It is enough to show that every element in $I(4,12)$ is included in one and only one of the three different types of sets on the  right hand side of the equality. Let $\alpha=
    (\alpha_1,\alpha_2,\alpha_3,\alpha_4)\in I(4,12)$. For $\alpha_i+ \alpha_j\neq 13$, where $i,j=1,2,3,4$, it follows that $(\alpha_1,\alpha_2,\alpha_3,\alpha_4)\in \Sigma
    \{\alpha_1,\alpha_2,\alpha_3,\alpha_4\}$. If for some $\alpha=(\alpha_1,\alpha_2,\alpha_3,\alpha_4)$ we have $\alpha_1 + \alpha_2 \neq 13$, without loss of generality we may assume that $\alpha_3 + \alpha_4 =13$, and then $\alpha \in \Sigma\{\alpha_1,\alpha_2\}$. Finally, if for some $\alpha=(\alpha_1,\alpha_2,\alpha_3,\alpha_4)$, $\alpha_1 + \alpha_2 = \alpha_3 + \alpha_4= 13$, then $\alpha \in C_2(\Sigma_6)$.
\end{proof}
\begin{rem}\label{rem1}
For each $1\leq \alpha_1<\alpha_2\leq 12$ such that $\alpha_1+\alpha_2\neq 13$, we have that $|\Sigma\{\alpha_1,\alpha_2\}|=4$ and $|\Sigma\{\alpha_1,\alpha_2\}: 1\leq \alpha_1<\alpha_2\leq 12|=60$. Hence, in the second term of the displayed expression in Lemma \ref{Part} there are $240$ indexes.
 Also, for each $1\leq \alpha_1<\alpha_2<\alpha_3<\alpha_4\leq 12$ such that $\alpha_i+\alpha_j\neq 13$, we have that $|\Sigma(\alpha_1,\alpha_2,\alpha_3,\alpha_4)|=1$, and $|\Sigma\{\alpha_1,\alpha_2,\alpha_3,\alpha_4\}: 1\leq \alpha_1<\alpha_2<\alpha_3<\alpha_4\leq 12|=240$. Hence, in the third term of the displayed expression in Lemma \ref{Part} there are $240$ indexes.
\end{rem}

We traslate now the combinatorial data of Lemma \ref{Part} in terms of the systems of linear equations associated to the contraction map $f:\wedge^6E \longrightarrow \wedge^4E$. For each $\alpha_{rs}\in I(4,12)$ consider the  linear  equation \eqref{Eqxx}. Now, for the part $C_2(\Sigma_6)$ in Lemma \ref{Part}, writing
    \begin{align*}
    C_2(\Sigma_6) & =
    \{(P_1,P_2),(P_1,P_3),(P_1,P_4),(P_1,P_5),(P_1,P_6),
    (P_2,P_3),(P_2,P_4),\\ & (P_2,P_5),(P_2,P_6),
    (P_3,P_4),(P_3,P_5),(P_3,P_6),(P_4,P_5),(P_4,P_6),(P_5,P_6)\}.
    \end{align*}
 For the set $C_2(\Sigma_6)$ ordered as above,  filling the symbols $\sqcup$ in \eqref{eq2.4} we obtain the system of linear equations
    \begin{align}\label{pi1}
    \Pi_{(P_1P_2)} & :=  X_{P_3 P_1P_2} + X_{P_4P_1P_2} + X_{P_5P_1P_2}
    + X_{P_6P_1P_2}=0\nonumber\\
    \Pi_{(P_1P_3)} & :=  X_{P_2 P_1P_3} + X_{P_4P_1P_3} + X_{P_5P_1P_3}
    + X_{P_6P_1P_3}=0\nonumber\\
    \Pi_{(P_1P_4)} & :=  X_{P_2 P_1P_4} + X_{P_3P_1P_4} + X_{P_5P_1P_4}
    + X_{P_6P_1P_4}=0\nonumber\\
    \Pi_{(P_1P_5)} & :=  X_{P_2 P_1P_5} + X_{P_3P_1P_5} + X_{P_4P_1P_5}
    + X_{P_6P_1P_5}=0\\
    \Pi_{(P_1P_6)} & :=  X_{P_2 P_1P_6} + X_{P_3P_1P_6} + X_{P_4P_1P_6}
    + X_{P_5P_1P_6}=0\nonumber\\
    \Pi_{(P_2P_3)} & :=  X_{P_1 P_2P_3} + X_{P_4P_1P_3} + X_{P_5P_1P_3}
    + X_{P_6P_1P_3}=0\nonumber\\
    \vdots & \quad \quad \quad \quad \quad \quad \quad  \vdots \nonumber\\
    \Pi_{(P_5P_6)} & :=  X_{P_1 P_5P_6} + X_{P_2P_5P_6} + X_{P_3P_5P_6}
    + X_{P_4P_5P_6}=0.\nonumber
    \end{align}
 where we identify the variable $X_{P_i P_jP_k}=X_{P'_i P'_j P'_k}$ if
 $\text{supp\,}\{P_i P_jP_k\}=\text{supp\,}\{P'_i P'_j P'_k\}$.

  Similarly, for the second part in the partition of $I(4,12)$ in Lemma \ref{Part}, for the sets $\Sigma
    \{\alpha_1,\alpha_2\}$, for each $1\leq\alpha_1<\alpha_2 \leq 12$, consider  the system of four homogeneous  linear equations $\Pi_{\alpha_{rs}}$ of \eqref{eq2.4}, for each $\alpha_{rs}\in\Sigma \{\alpha_1\alpha_2\}$, which have the form:
  \begin{align}\label{pi2}
    X_1 + X_2 + X_3  = 0\nonumber\\
    X_1 + X_4 + X_5  = 0 \\
    X_2 + X_4 + X_6  = 0\nonumber\\
    X_3 + X_5 + X_6  = 0.\nonumber
  \end{align}
and  there are $60$ such systems of linear equations \eqref{pi2}.
For example, for $\Sigma\{1,2\}=\{12P_3,12P_4,12P_5,12P_6\}$,  setting $A=10$, $B=11$ and $C=12$, as in \eqref{eq2.4},
the system \eqref{pi2} is
  \begin{align*}
 \Pi_{12P_3}&:=   X_{412P_39} + X_{512P_38} + X_{612P_37}  = 0 \\
  \Pi_{12P_4}&:=    X_{312P_4A} + X_{512P_48}+ X_{612P_47}  = 0 \\
  \Pi_{12P_5}&:=    X_{312P_5A} + X_{412P_59} + X_{612P_57}  = 0 \\
  \Pi_{12P_6}&:=    X_{312P_6A} + X_{412P_69} + X_{512P_68}  = 0.
  \end{align*}

 Finally, for the third part in the partition of $I(4,12)$ in Lemma \ref{Part}, for each set $\Sigma\{\alpha_1,\alpha_2,\alpha_3,\alpha_4\}$, with $1\leq\alpha_1<\alpha_2<\alpha_3<\alpha_4 \leq 12$, the matrix ${\EuScript L}_2$ of the corresponding linear equation
    $\Pi_{\alpha_1,\alpha_2,\alpha_3,\alpha_4}$ of \eqref{eq2.4} is a matrix (row vector)
    with the $P_{i,\alpha_1,\alpha_2,\alpha_3,\alpha_4}$ and
    $P_{j,\alpha_1,\alpha_2,\alpha_3,\alpha_4}$ components equal to one
    and all the other components equal to zero for $1\leq i< j \leq 6$. There are $240$ such matrices ${\EuScript L}_2$.
    The size of this vector is $1\times C^{12}_6$.  For example,
$$\Pi_{1234}=X_{512348}+X_{612347}=0$$
with corresponding matrix ${\EuScript L}_2:=(0,\ldots,0,1,0,\ldots,0,1,0,\ldots,0)$, where the first $1$ is in the coordinate corresponding to $123458$ and the second $1$ is in the coordinate correponding to the $123467$.

The coefficient matrix associated to the system   \eqref{pi1}  is
{\label{l3-4}}
$$ {\EuScript L}_4 = \left( \begin{tabular}{cccccccccccccccccccc}
1 & 1 & 1 & 1 & 0 & 0 & 0 & 0 & 0 & 0 & 0 & 0 & 0 & 0 & 0 & 0 & 0 &
0 & 0 & 0 \cr \cline{1-4}\multicolumn{1}{|c} 1 & 0 & 0 & 0 &
\multicolumn{1}{|c}{$1$} & 1 & 1 & 0 & 0 & 0 & 0 & 0 & 0 & 0 & 0 & 0
& 0 & 0 & 0 & 0 \cr \cline{5-7}\multicolumn{1}{|c} 0 & 1 & 0 & 0 &
\multicolumn{1}{|c}{$1$} & 0 &  0 & \multicolumn{1}{|c} 1 & 1 & 0 &
0 & 0 & 0 & 0 & 0 & 0 & 0 & 0 & 0 & 0 \cr
\cline{8-9}\multicolumn{1}{|c} 0 & 0 & 1 & 0 & \multicolumn{1}{|c} 0
& 1 & 0 & \multicolumn{1}{|c} 1 & 0 & \multicolumn{1}{|c} 1 & 0 & 0
& 0 & 0 & 0 & 0 & 0 & 0 & 0 & 0 \cr \cline{10-10}\multicolumn{1}{|c}
0 & 0 & 0 & 1 & \multicolumn{1}{|c} 0 & 0 &  1 & \multicolumn{1}{|c}
0 & 1 & \multicolumn{1}{|c} 1 & \multicolumn{1}{|c} 0 & 0 & 0 & 0 &
0 & 0 & 0 & 0 & 0 & 0 \cr \cline{1-10}\multicolumn{1}{|c} 1 & 0 & 0
& 0 & 0 & 0 & 0 & 0 & 0 & 0 & \multicolumn{1}{|c} 1 & 1 & 1 & 0 & 0
& 0 & 0 & 0 & 0 & 0 \cr \cline{11-13}\multicolumn{1}{|c} 0 & 1 & 0 &
0 & 0 & 0 & 0 & 0 & 0 & 0 &\multicolumn{1}{|c} 1 & 0 & 0 &
\multicolumn{1}{|c} 1 & 1 & 0 & 0 & 0 & 0 & 0 \cr
\cline{14-15}\multicolumn{1}{|c} 0 & 0 & 1 & 0 & 0 & 0 & 0 & 0 & 0 &
0 & \multicolumn{1}{|c} 0 & 1 & 0 & \multicolumn{1}{|c} 1 & 0 &
\multicolumn{1}{|c} 1 & 0 & 0 & 0 & 0 \cr \cline{16-16}
\multicolumn{1}{|c} 0 & 0 & 0 & 1 & 0 & 0 & 0 & 0 & 0 & 0 &
\multicolumn{1}{|c} 0 & 0 & 1 & \multicolumn{1}{|c} 0 & 1 &
\multicolumn{1}{|c} 1 & \multicolumn{1}{|c} 0 & 0 & 0 & 0 \cr
\cline{11-16}\multicolumn{1}{|c} 0 & 0 & 0 & 0 & 1 & 0 & 0 & 0 & 0 &
0 & \multicolumn{1}{|c} 1 & 0 & 0 & 0& 0 & 0 & \multicolumn{1}{|c} 1
& 1 & 0 & 0 \cr \cline{17-18} \multicolumn{1}{|c} 0 & 0 & 0 & 0 &  0
& 1 & 0 & 0 & 0 & 0 &\multicolumn{1}{|c}  0 & 1 & 0 & 0& 0 & 0 &
\multicolumn{1}{|c} 1 & 0 & \multicolumn{1}{|c} 1 & 0 \cr
\cline{19-19}\multicolumn{1}{|c} 0 & 0 & 0 & 0 & 0 & 0 & 1 & 0 & 0 &
0 & \multicolumn{1}{|c} 0 & 0 & 1 & 0& 0 & 0 & \multicolumn{1}{|c} 0
& 1 & \multicolumn{1}{|c} 1 & \multicolumn{1}{|c} 0 \cr
\cline{17-19}\multicolumn{1}{|c} 0 & 0 & 0 & 0 & 0 & 0 & 0 & 1 & 0 &
0 & \multicolumn{1}{|c} 0 & 0 &  0 & 1& 0 & 0 & \multicolumn{1}{|c}
1 & 0 & 0 & \multicolumn{1}{|c} 1 \cr
\cline{20-20}\multicolumn{1}{|c} 0 & 0 & 0 & 0 & 0 & 0 & 0 & 0 & 1 &
0 & \multicolumn{1}{|c} 0 & 0 & 0 & 0& 1 & 0 & \multicolumn{1}{|c} 0
& 1 & 0 & \multicolumn{1}{|c|} 1 \cr
\cline{20-20}\multicolumn{1}{|c} 0 & 0 & 0 & 0 & 0 & 0 & 0 & 0 & 0 &
1 & \multicolumn{1}{|c} 0 & 0 & 0 & 0& 0 & 1 & \multicolumn{1}{|c} 0
& 0 & 1 & \multicolumn{1}{|c|} 1 \cr \cline{1-20}
\end{tabular} \right) \quad $$
and the corresponding matrix associated to the system \eqref{pi2} is
$${\EuScript L}_3 =\begin{pmatrix}
1 & 1 & 1 & 0 & 0 & 0 \\
1 & 0 & 0 & 1 & 1 & 0 \\
0 & 1 & 0 & 1 & 0 & 1 \\
0 & 0 & 1 & 0 & 1 & 1
\end{pmatrix}.$$
For the matrix ${\EuScript L}_4$, adding its 14 last rows to the first row, we obtain that
${\EuScript L}_4$ is row-equivalent to the matrix $\begin{pmatrix}
\overline{3}\\
\widehat{{\EuScript L}}_4
\end{pmatrix}$, where $\widehat{{\EuScript L}}_4$ is the matrix obtained from ${\EuScript L}_4$ deleting its first row, and $\overline{3}$ is a row all whose entries are equal to $3$. Thus, if $\text{\rm char}({\mathbb F})=3$, the rank of ${\EuScript L}_4$ is $\leq 14$. Moreover, a direct computation shows that its rank is indeed $14$ in characteristic $3$.
Similarly, for the matrix ${\EuScript L}_3$, adding the last three rows to the first one we see that if $\text{\rm char}({\mathbb F})=2$ then the rank of ${\EuScript L}_3$ is $\leq 3$, and again a computation shows that is exactly $3$.
Moreover,  $\text{\rm rank}({\,\EuScript L}_3)=4$ if
$\text{\rm char}({\mathbb F})\neq 2$,  and
     $\text{\rm rank}({\,\EuScript L}_4) =15$ if $\text{\rm char}({ \mathbb F}) \neq 2,3$.

\begin{prop}\label{MatB}
For any field ${\mathbb F}$, the matrix $B$, of size $C^{12}_4\times C^{12}_6$, associated to the homogeneous system
    $\Pi=\{ \Pi_{\alpha_{rs}} | \alpha_{rs}\in I(4,12)\}$ can be given by a
    block diagonal matrix  as follows
    \begin{align*}
    B  &={\EuScript L}_4 \oplus \Big( \bigoplus_{1\leq \alpha_1 <
        \alpha_2 \leq 12 \atop \alpha_1 + \alpha_2 \neq 13 } {\EuScript
        L}_3^{(\alpha_1, \alpha_2)} \Big) \oplus \Big(\bigoplus_{1\leq
        \alpha_1< \alpha_2 <\alpha_3 < \alpha_4 \leq 12 \atop \alpha_i +
        \alpha_j \neq 13 }
    {\EuScript L}_{2}^{(\alpha_1, \alpha_2, \alpha_3, \alpha_4)} \Big),\nonumber \\
    B &=
    \begin{pmatrix}
    \begin{tabular}{|c c c c c c c}
    \cline{1-1}\multicolumn{1}{|c|}{${\EuScript L}_4$}  & & & & \multicolumn{1}{c}{ } &    \\
    \cline{2-1}\cline{1-1} \multicolumn{1}{c}{}  &
    \multicolumn{1}{|c|}{${\EuScript L}_3$} & & &
    \multicolumn{1}{c}{$0$}& \\ 
    \cline{2-2}\multicolumn{1}{c}{} & \multicolumn{1}{c}{} &
    \multicolumn{1}{c}{$ \ddots$} & & \multicolumn{1}{c}{ }  & \\ 
    \cline{4-3}\multicolumn{1}{c}{ } &\multicolumn{1}{c}{$0$ } &
    \multicolumn{1}{c|}{ }&\multicolumn{1}{c|}{${\EuScript L}_3$}  & & \\
    \cline{5-3}\cline{4-4} \multicolumn{1}{c}{} & \multicolumn{1}{c}{}
    & \multicolumn{1}{c}{}& \multicolumn{1}{c|}{} &
    \multicolumn{1}{c|}{${\EuScript L}_2$}   & \\
    \cline{5-5} \multicolumn{1}{c}{} & \multicolumn{1}{c}{} &
    \multicolumn{1}{c}{}& \multicolumn{1}{c}{} & \multicolumn{1}{c}{}
    & \multicolumn{1}{c}{$\ddots$}
    \\ \cline{7-7}
    \multicolumn{1}{c}{} & \multicolumn{1}{c}{} &
    \multicolumn{1}{c}{}& \multicolumn{1}{c}{ } &
    \multicolumn{1}{c}{ } & \multicolumn{1}{c|}{ } &   \multicolumn{1}{c|}{${\EuScript L}_2$} \\
    \cline{7-7}
    \end{tabular} \end{pmatrix},
    \end{align*}
    where there are $1$ matrix ${\EuScript L}_4$, $60$ submatrices ${\EuScript L}_3$,  and $240$ submatrices ${\EuScript L}_2$.
\end{prop}

\begin{proof} It follows from the observation that $I(4,12)$ is a disjoint union of the
sets described in Lemma \ref{Part}   and the one-to-one relationship between those sets and their corresponding system of homogeneous  linear equations.
\end{proof}

For the contraction map $f:\wedge^6 E \longrightarrow \wedge^4E$ given by $(\ref{mapf})$, we obtain,
from Proposition \ref{MatB}, the following consequences:
\begin{enumerate}
\item[(1)] If  ${\rm char}(\mathbb{F})$$\,= 3$, then ${\rm rank}(B) = {\rm rank}({\EuScript L}_4) +60\, {\rm rank}({\EuScript L}_3) +240\,{\rm rank}({\EuScript L}_2) = 494$.
\item[(2)]  $\dim_{\mathbb{F}}(\ker f) = C_{6}^{12}-494 = 430$.
\item[(3)] If {\rm char}$({\mathbb{F}})=3$, then the map $f$ is not surjective.
\end{enumerate}

From Proposition \ref{MatB} we calculate the codimension of the linear
section ${\mathbb P}(\ker f)$ in ${\mathbb P}(\nwedge^6E)$ for any characteristic. That is
\begin{center}
\begin{tabular}{|c|c|c|}
\cline{1-3}$\text{char}(F)$  & $\text{rank}(B)$ & \text{codimension of} ${\mathbb P}(\ker f)$ \\
\cline{1-3} 0 & 495 & 429 \\
\cline{1-3} 2 & 430  & 494 \\
\cline{1-3}3  & 494 & 430 \\
\cline{1-3} $ p \geq 5$   & 495 & 429  \\
\cline{1-3}
\end{tabular}
\end{center}

This computation shows that the codimension of ${\mathbb P}(\ker f)$ in ${\mathbb P}(\nwedge^nE)$ depends on the dimension $2n$ of the symplectic space $E$ and the characteristic of the ground field ${\mathbb F}$. In a forthcoming paper the authors show that, in general, the contraction map $f$  is surjective if and only if $\text{char}({\mathbb F})=0$ or $\text{char}({\mathbb F})\geq m$, for a certain integer $m$.


\begin{thebibliography}{9}

\bibitem{bib1} Carrillo-Pacheco, J, Jarqu\'in-Z\'arate, F, Velasco-Fuentes,
M and Zald\'{\i}var, F,  An explicit description in terms of
Pl\"ucker coordinates of the Langrangian-Grassmannian. arXiv/1601.07501. Preprint (2016).

\bibitem{bib2} J. Carrillo-Pacheco and F.  Zaldivar,   On Lagrangian-Grassmannian Codes. Designs, Codes and Cryptography  60 (2011)  291--268.

\end{thebibliography}
\end{document}